\documentclass{article}%
\usepackage{amsmath}
\usepackage{amsfonts}
\usepackage{amssymb}
\usepackage{graphicx}%
\providecommand{\U}[1]{\protect\rule{.1in}{.1in}}
\newtheorem{theorem}{Theorem}

\newtheorem{lemma}[theorem]{Lemma}

\newtheorem{proposition}[theorem]{Proposition}

\newenvironment{proof}[1][Proof]{\noindent\textbf{#1.} }{\ \rule{0.5em}{0.5em}}
\begin{document}

\author{Nikolay Nikolov and Dan Segal}
\title{Constructing uncountably many groups with the same profinite completion}
\maketitle

To what extent is a finitely generated residually finite (f.g.r.f.) group
determined by its finite quotients? This question can be formulated in various
ways, see for example \cite{GZ}. The family of all finite quotients of a group
$G$ is determined by its inverse limit, the profinite completion $\widehat{G}%
$; following \cite{GZ} let us define the \emph{genus} of an f.g.r.f. group $G$
to be the set of isomorphism classes of f.g.r.f. groups $H$ such that
$\widehat{H}\cong\widehat{G}$. If $G$ is abelian, the genus is a singleton; if
$G$ is nilpotent, the genus is finite, a deep result of P. F. Pickel
\cite{P1}; if $G$ is metabelian the genus can be countably infinite \cite{P2}.

Uncountable genera (in fact uncountably many such) were first constructed by
Pyber \cite{P}: in that case the finite quotients are products of distinct
alternating groups. The only other examples (to our knowledge) are due to
Nekrashevych \cite{N}: here the finite quotients are $2$-groups.

It struck us that the constructions introduced in \cite{KKN} and in \cite{S}
could be adapted to yield uncountable genera.

Those of the first kind are \emph{soluble}: indeed this is the first example
of an uncountable genus of f.g. soluble groups. Our groups have derived length
four. They couldn't be metabelian, like Pickel's groups, because there are
only countably many f.g. metabelian groups; in fact our proof is more
elementary than Pickel's approach, which depends on the theory of Picard
groups. Whether a genus of f.g. soluble groups of derived length $3$ could be
uncountable seems an interesting question.

Those of the second kind, like Nekrashevych's groups, are \emph{branch
groups}. The method is easier than his, however: using perfect groups in place
of $2$-groups gives one cheap access to the relevant `congruence subgroup
property' (explained below).

Both constructions actually yield uncountably many distinct uncountable
genera; we shall not spell this out but it is implicit in the proofs.

\section{Soluble groups}

Let $G$ be the permutational wreath product $C_{2}\wr C_{2}\wr C_{\infty}$, a
three-generator soluble group of derived length $3$. We shall use the (easy)
fact that $G$ is residually finite.

The key result is

\begin{proposition}
\label{mod}There is a family of $2^{\aleph_{0}}$ pairwise non-isomorphic,
residually finite cyclic $\mathbb{Z}G$-modules all having the same finite images.
\end{proposition}

If these modules are $M_{\alpha},$ $\alpha\in X,$ the corresponding split
extensions $\Gamma_{\alpha}=M_{\alpha}\rtimes G$ all have isomorphic profinite
completions. They are all quotients of%
\[
\Gamma^{\ast}=\mathbb{Z}G\rtimes G=\mathbb{Z}\wr G;
\]
if $M_{a}=\mathbb{Z}G/J_{\alpha}$ then $\Gamma_{\alpha}\cong\Gamma^{\ast
}/K_{\alpha}$ where $K_{\alpha}=J_{\alpha}.1<$ $\mathbb{Z}G\rtimes G$. For
each $\alpha\in X$ the set of $\beta\in X$ such that $\Gamma_{\beta}%
\cong\Gamma_{\alpha}$ is countable, since there are only countably many
epimorphisms from the $4$-generator group $\Gamma^{\ast}$ to the countable
group $\Gamma_{\alpha}$. It follows that the groups $\Gamma_{\alpha}$,
$\alpha\in X$ lie in $2^{\aleph_{0}}$ isomorphism classes. Thus we may infer

\begin{theorem}
\label{genus} There are $2^{\aleph_{0}}$ pairwise non-isomorphic $4$-generator
residually finite soluble groups of derived length $4$ with the same finite
images. They are all quotients of $\mathbb{Z}\wr(C_{2}\wr C_{2}\wr C_{\infty
})$.
\end{theorem}

Let us set up some notation. Let $V$ be a vector space over $\mathbb{F}_{2}$
with basis $\{e_{i},f_{i}\ |\ i\in\mathbb{Z}\}$. Let $a\in\mathrm{GL}(V)$ be
the automorphism which swaps $e_{0}$ with $f_{0}$ and fixes the other basis
vectors. Let $t\in\mathrm{GL}(V)$ be the automorphism such that $e_{i}%
t=e_{i+1}$ and $f_{i}t=f_{i+1}$ for each $i\in\mathbb{Z}$. Then $\langle
a,t\rangle$ generate a copy of $C_{2}\wr C_{\infty}$ in $\mathrm{GL}(V)$ and
we identify $G$ with $V\rtimes\langle a,t\rangle\leq V\rtimes\mathrm{GL}(V)$.
Note that $G$ contains the elementary abelian subgroup $\langle a_{i}%
~\mid~i\in\mathbb{Z}\rangle$ where $a_{i}=a^{t^{i}}\in\mathrm{GL}(V)$ is the
automorphism of order $2$ which swaps $e_{i}$ with $f_{i}$ and fixes the other
basis vectors.

For $\lambda\in Y:=\{0,1\}^{\mathbb{N}}$ we define the sequence $\mathbf{c}%
_{\lambda}=(c_{i})_{i\in\mathbb{N}}$ by%
\begin{align*}
c_{2n-1}  &  =e_{n},~~c_{2n}=f_{n}~\text{if }\lambda(n)=0\\
c_{2n-1}  &  =f_{n},~~c_{2n}=e_{n}~\text{if }\lambda(n)=1,
\end{align*}
and an ascending chain of subgroups of $V$ by
\[
H_{\lambda,i}=\langle e_{0},f_{0},e_{-1},f_{-1},\ldots,e_{-i},f_{-i}%
,c_{1},\ldots,c_{i}\rangle.
\]

The following is then clear:

\begin{lemma}
\label{chains}\emph{(i)} $V=\bigcup_{i=1}^{\infty}H_{\lambda,i}$.

\emph{(ii)} For each $\alpha,$ $\beta\in Y$ and each $n\in\mathbb{N}$ there is
an element $g=g(\alpha,\beta,n)\in\langle a_{1},\ldots,a_{n}\rangle<G$ such
that $H_{\alpha,i}^{g}=H_{\beta,i}$ for $i=1,2,\ldots,n$.
\end{lemma}

Fix an infinite sequence of distinct primes $(p_{i})_{i\in\mathbb{N}}$. We now
define a $\mathbb{Z}G$-module $M_{\lambda}$ for each $\lambda\in Y$. For
clarity, the subscript $\lambda$ will sometimes be suppressed. Let
\[
U_{\lambda,i}=U_{i}=\mathbb{F}_{p_{i}}G/(H_{i}-1)\mathbb{F}_{p_{i}}%
G\cong\mathbb{F}_{p_{i}}\otimes_{\mathbb{F}_{p_{i}}H_{i}}\mathbb{F}_{p_{i}}G;
\]
this is the right permutation $\mathbb{F}_{p_{i}}G$ module on the right cosets
$\{H_{i}g\ |\ g\in G\}$ of $H_{i}$ in $G$, and we fix the module generator
\[
u_{i}=1+(H_{i}-1)\mathbb{F}_{p_{i}}G.
\]
Now $M_{\lambda}$ is defined to be the cyclic $\mathbb{Z}G$-submodule of
$\prod_{i=1}^{\infty}U_{i}$ generated by $\mathbf{u}_{\lambda}=(u_{1}%
,u_{2},\ldots)$. Thus%
\[
M_{\lambda}\cong\mathbb{Z}G/J_{\lambda}%
\]
where%
\begin{align*}
J_{\lambda}  &  =\mathrm{ann}_{\mathbb{Z}G}(\mathbf{u}_{\lambda})\\
&  =\bigcap\limits_{i=1}^{\infty}\left(  (H_{\lambda,i}-1)\mathbb{Z}%
G+p_{i}\mathbb{Z}G\right)  .
\end{align*}

Since each $H_{i}$ is finite and $G$ is residually finite, $H_{i}$ is closed
in the profinite topology of $G$. Thus the annihilator in $\mathbb{F}_{p_{i}%
}G$ of $u_{i},$ namely $(H_{i}-1)\mathbb{F}_{p_{i}}G,$ is the intersection of
finite-codimension right ideals of the form $(L-1)\mathbb{F}_{p_{i}}G$ (here
$L$ ranges over the subgroups of finite index in $G$ that contain $H_{i}$);
hence each $U_{i}$ is residually finite as a $G$-module. It follows that
$\prod_{i=1}^{\infty}U_{i}$ is also a residually finite $G$-module, and then
so is its submodule $M_{\lambda}$.

\begin{lemma}
\label{isoquot}Let $\gamma,~\beta\in Y$ and let $N=N^{G}\leq V$. If
$NH_{\gamma,i}=NH_{\beta,i}$ \ for all $i$ then the $G$-modules $M_{\gamma
}/M_{\gamma}(N-1)$ and $M_{\beta}/M_{\beta}(N-1)$ are isomorphic.
\end{lemma}

\begin{proof}
It will suffice to prove the (stronger) statement%
\begin{equation}
J_{\gamma}+(N-1)\mathbb{Z}G=J_{\beta}+(N-1)\mathbb{Z}G. \label{annihilators}%
\end{equation}

Let $x\in N.$ Then for some $k$ we have $x\in H_{\gamma,i}$ and $x\in
H_{\beta,i}$ for all $i>k.$ Then $u_{\gamma,i}(x-1)=0$ and $u_{\beta
,i}(x-1)=0$ for every $i>k$. It follows that for $\lambda=\gamma,~\beta,$%
\[
\mathbf{u}_{\lambda}(x-1)=\sum\nolimits_{i=1}^{k}u_{\lambda,i}(x-1)\in
\bigoplus\nolimits_{i}U_{\lambda,i}(N-1)<\prod_{i=1}^{\infty}U_{\lambda,i}.
\]
Thus $J_{\lambda}+(N-1)\mathbb{Z}G$ maps $\mathbf{u}_{\lambda}$ into
$\bigoplus\nolimits_{i}U_{\lambda,i}(N-1)=D,$ say.

Let $I_{\lambda}$ denote the annihilator in $\mathbb{Z}G$ of $\mathbf{u}%
_{\lambda}$ modulo $D$. Suppose $z\in$ $I_{\lambda}$. Then
\[
\mathbf{u}_{\lambda}z=(u_{\lambda,i}s_{i})_{i}%
\]
with each $s_{i}\in(N-1)\mathbb{Z}G,$ and $s_{j}=0$ for all $j>m$, say. By the
preceding paragraph, there exists $k$ such that $u_{\lambda,i}s_{j}=0$ for
each $j\leq m$ and all $i>k$. Now we choose integers $q_{i}$ such that
$q_{i}\equiv\delta_{ij}~(\operatorname{mod}p_{j})$ for $i,j=1,\ldots,k$.
Taking $r=$ $\sum\nolimits_{i=1}^{k}q_{j}s_{j}$ we have%
\begin{align}
u_{\lambda,i}r  &  =u_{\lambda,i}s_{i}~~\text{if }i\leq k \label{localcondns1}%
\\
u_{\lambda,i}r  &  =0~\text{if }i>k. \label{condsn2}%
\end{align}
Thus $\mathbf{u}_{\lambda}z=\mathbf{u}_{\lambda}r$ so $z\in J_{\lambda
}+r\subseteq J_{\lambda}+(N-1)\mathbb{Z}G$.

It follows that $I_{\lambda}=J_{\lambda}+(N-1)\mathbb{Z}G.$ Thus it remains to
show that $I_{\gamma}=I_{\beta}.$

Now let $r\in I_{\gamma}$. Then (\ref{localcondns1}) and (\ref{condsn2}) hold
(with $\gamma$ for $\lambda)$, for some $k$ and some $s_{i}\in(N-1)\mathbb{Z}%
G.$

(\ref{condsn2}) is equivalent to%
\[
r\in(H_{\gamma,i}-1)\mathbb{Z}G+p_{i}\mathbb{Z}G~\ ~~\forall i>k.
\]
This implies%
\[
r\in\bigcap\limits_{i>k}\left(  (V-1)\mathbb{Z}G+p_{i}\mathbb{Z}G\right)
=(V-1)\mathbb{Z}G
\]
which in turn implies that for some $k_{1}$ we have%
\begin{equation}
r\in(H_{\beta,i}-1)\mathbb{Z}G~~\forall i>k_{1}. \label{k-one}%
\end{equation}

In (\ref{localcondns1}), we may enlarge $k$ arbitrarily by setting $s_{i}=0$
for finitely many values of $i>k$; so we may assume that $k\geq k_{1}$. Now
(\ref{localcondns1}) is equivalent to%
\begin{align*}
r  &  \in(H_{\gamma,i}-1)\mathbb{Z}G+(N-1)\mathbb{Z}G+p_{i}\mathbb{Z}G\\
&  =(NH_{\gamma,i}-1)\mathbb{Z}G+p_{i}\mathbb{Z}G\\
&  =(NH_{\beta,i}-1)\mathbb{Z}G+p_{i}\mathbb{Z}G~\ \ (i\leq k).
\end{align*}
Together with (\ref{k-one}), this shows that (\ref{localcondns1}) and
(\ref{condsn2}) hold with $\beta$ for $\lambda$, and so $r\in I_{\beta}$. The
result follows by symmetry.
\end{proof}

\bigskip

Now fix $\alpha,~\beta\in Y$. For any $\lambda\in Y$, every finite $G$-module
image of $M_{\lambda}$ is an image of $M_{\lambda}/M_{\lambda}(N-1)$ for some
subgroup $N=N^{G}$ of finite index in $V$. There exists $k$ such that
$H_{\alpha,i}N=H_{\beta,i}N=V$ for all $i>k,$ and there exists $g=g(\alpha
,\beta,k)\in G$ such that $H_{\alpha,i}^{g}=H_{\beta,i}$ for $1\leq i\leq k$.
We can specify $\gamma\in Y$ so that $H_{\gamma,i}=H_{\alpha,i}^{g}$ for all
$i$. Then $H_{\gamma,i}N=H_{\beta,i}N$ for all $i$, and Lemma \ref{isoquot}
gives%
\[
\frac{M_{\gamma}}{M_{\gamma}(N-1)}\cong\frac{M_{\beta}}{M_{\beta}(N-1)}.
\]
On the other hand,%
\[
M_{\gamma}\cong\frac{\mathbb{Z}G}{J_{\gamma}}=\frac{\mathbb{Z}G}%
{g^{-1}J_{\alpha}}\cong\frac{\mathbb{Z}G}{J_{\alpha}}\cong M_{\alpha}.
\]

It follows that $M_{\alpha}/M_{\alpha}(N-1)\cong M_{\beta}/M_{\beta}(N-1)$. We
infer that $M_{\alpha}$ and $M_{\beta}$ have the same finite images as $G$-modules.

\begin{lemma}
\label{ann} The map $\lambda\longmapsto J_{\lambda}$ ($\lambda\in Y$) is bijective.
\end{lemma}

\begin{proof}
It suffices to show that for each $n$, $\lambda(n)$ is determined by
$J_{\lambda}$. Now fix $n$ and set $i=2n-1.$ Then
\begin{align*}
\lambda(n)=0  &  \Longleftrightarrow e_{n}\in H_{\lambda,i}\\
&  \Longleftrightarrow p_{1}p_{2}\cdots p_{i-1}(e_{n}-1)\in J_{\lambda}.
\end{align*}
To see this, observe that if $g\in H_{\lambda,i}$ then $u_{\lambda,j}(g-1)=0$
for all $j\geq i$ and $p_{1}p_{2}\cdots p_{i-1}u_{j}=0$ for all $j<i$; while
if $g\in G\smallsetminus H_{\lambda,i}$ then $u_{\lambda,i}g\neq u_{\lambda
,i}$ so $u_{\lambda,i}.p_{1}p_{2}\cdots p_{i-1}(g-1)\neq0$ since $p_{1}%
p_{2}\cdots p_{i-1}$ is invertible in $\mathbb{F}_{p_{i}}$.
\end{proof}

\bigskip

Now given $\alpha\in Y$, the set of $\beta\in Y$ such that $\mathbb{Z}%
G/J_{\beta}\cong\mathbb{Z}G/J_{\alpha}$ is countable, since for each such
$\beta$ there exists an epimorphism from $\mathbb{Z}G$ onto the countable
module $\mathbb{Z}G/J_{\alpha}$ with kernel $J_{\beta}$. As $\left\vert
Y\right\vert =2^{\aleph_{0}}$, Lemma \ref{ann} ensures that the modules
$M_{\alpha}\cong\mathbb{Z}G/J_{\alpha}$ lie in $2^{\aleph_{0}}$ isomorphism
classes, and Proposition \ref{mod} follows.

\section{\bigskip Branch groups}

For details of the following construction, see \cite{S}, \S 2 or \cite{LS},
\S 13.4. We start with a rooted tree $\mathcal{T}$, in which each vertex of
level $n\geq1$ has valency $1+l_{n}$ (and the root has valency $l_{0}$). For
each $n$ we take a permutation group $T_{n}$ of degree $l_{n},$ set
$W_{0}=T_{0},$ and for $n\geq0$ let $W_{n+1}=T_{n}\wr W_{n-1}$ be the
permutational wreath product. This acts in a natural way on the finite tree
$\mathcal{T}[n+1]$ obtained by truncating $\mathcal{T}$ at level $n+1$. Hence
the inverse limit\bigskip%
\[
W=~\underset{\leftarrow}{\lim}W_{n}%
\]
sits naturally as a subgroup of $\mathrm{Aut}(\mathcal{T})$.

Now $W$ is a profinite group, a base for the neighbourhoods of the identity
being the set of level stabilizers
\[
\mathrm{St}_{W}(n)=\ker\left(  W\rightarrow\mathrm{Aut}(\mathcal{T}%
[n])\right)  .
\]
A subgroup $\Gamma$ of $W$ is said to have the \emph{congruence subgroup
property} if the natural topology of $W$ induces the profinite topology on
$\Gamma$, that is, if every subgroup of finite index in $\Gamma$ contains
$\mathrm{St}_{\Gamma}(n)=\Gamma\cap\mathrm{St}_{W}(n)$ for some $n$. If this
holds, then the natural homomorphism $\widehat{\Gamma}\rightarrow W$ is
injective; if in addition $\Gamma$ is \emph{dense} in $W,$ it follows that
$\widehat{\Gamma}\cong W$.

On pages 262-263 of \cite{LS} we define four elements $\xi,$ $\eta,~a$ and $b$
of $W,$ set $\Gamma=\left\langle \xi,\eta,a,b\right\rangle $, and prove that
under certain conditions, $\Gamma$ is both dense and satisfies the congruence
subgroup property.

The conditions are as follows:

(i) $T_{n}$ is a doubly transitive subgroup of $\mathrm{Sym}(l_{n})$ (this
condition can be considerably weakened);

(ii) there exist a two-generator perfect group $P=\left\langle
x,y\right\rangle $ and for each $n$ an epimorphism $\phi_{n}:P\rightarrow
T_{n};$

(iii) the automorphisms $\xi,$ $\eta,~a$ and $b$ are built in a particular way
out of the%
\[
\alpha_{n}=x\phi_{n},~\beta_{n}=y\phi_{n}\in T_{n}\leq\mathrm{Sym}%
(l_{n})\text{.}%
\]
Specifically, $\xi$ and $\eta$ are `rooted automorphisms', permuting bodily
the $l_{0}$ subtrees attached to the root of $\mathcal{T}$ as $\alpha_{0},$
$\beta_{0}$ respectively; $a$ and $b$ are so-called `directed' (or `spinal')
automorphisms corresponding to the sequences $(\alpha_{n}),$ $(\beta_{n})$ --
these act as rooted automorphisms on subtrees of $\mathcal{T}$ rooted at
vertices one step away from a fixed infinite path (pictured on page 262 of
\cite{LS}).

In principle we could choose any sequence of finite simple groups; for
simplicity let us assume that $l_{n}=5$ and $T_{n}=\mathrm{Alt}(5)$ for all
$n$. Put%
\[
\alpha=(123),~\beta=(12345).
\]
Let $\lambda\in\{0,1\}^{\mathbb{N}_{0}}$ and set%
\begin{align*}
\alpha_{n}  &  =\alpha,\beta_{n}=\beta\text{ if }\lambda_{n}=0\\
\alpha_{n}  &  =\beta,\beta_{n}=\alpha\text{ if }\lambda_{n}=1.
\end{align*}
Let $P=\mathrm{Alt}(5)\times\mathrm{Alt}(5),$ \thinspace$x=(\alpha,\beta),$
$y=(\beta,\alpha)\in P$. Then $P=\left\langle x,y\right\rangle ,$ and we
define $\phi_{n}:P\rightarrow T_{n}$ by%
\begin{align*}
x\phi_{n}  &  =\alpha^{1-\lambda_{n}}\cdot\beta^{\lambda_{n}}\\
y\phi_{n}  &  =\alpha^{\lambda_{n}}\cdot\beta^{1-\lambda_{n}},
\end{align*}
thus $\phi_{n}$ is simply the projection of $P=\mathrm{Alt}(5)\times
\mathrm{Alt}(5)$ onto either the first or the second direct factor.

Let $\Gamma(\lambda)=\left\langle \xi(\lambda),\eta(\lambda),a(\lambda
),b(\lambda)\right\rangle $ denote the group $\Gamma$ constructed as above
using the sequence $\lambda$. There are $2^{\aleph_{0}}$ such sequences, so we
have constructed $2^{\aleph_{0}}$ $4$-generator subgroups of $\mathrm{Aut}%
(\mathcal{T})$ with profinite completion $W$. (These groups are of course
residually finite since $\mathrm{Aut}(\mathcal{T})$ is.)

\bigskip\ \ \ \ \ \ \ \ 

\textbf{Claim:} \emph{For each sequence} $\lambda$\emph{, the set}
$S(\lambda):=\{\mu\mid$ $\Gamma(\lambda)\cong\Gamma\left(  \mu\right)  \}$
\emph{is countable.}

\bigskip

The claim implies that the number of isomorphism classes among the groups
$\Gamma(\lambda)$ is still $2^{\aleph_{0}}$, and yields

\begin{theorem}
There are continuously many pairwise non-isomorphic $4$-generator residually
finite groups all having the iterated wreath product $W$ as their profinite completion.
\end{theorem}

To establish the claim, we suppose that $S(\lambda)$ is uncountable, and aim
to derive a contradiction.\ 

For $\mu\in S(\lambda)$ let $\theta_{\mu}:\Gamma(\mu)\rightarrow\Gamma\left(
\lambda\right)  $ be an isomorphism. Then $\theta_{\mu}$ extends to a
continuous automorphism $\sigma_{\mu}$ of $W$ (universal property of profinite completions).

Now the set%
\[
\left\{  a(\mu)^{\sigma}\mid\mu\in S(\lambda),~\sigma\in\mathrm{Aut}%
(W)\right\}  \subseteq\Gamma\left(  \lambda\right)
\]
is countable because $\Gamma\left(  \lambda\right)  $ is a finitely generated
group. Hence there exists $c\in\Gamma\left(  \lambda\right)  $ such that the
set%
\[
X:=\left\{  \mu\in S(\lambda)\mid a(\mu)^{\sigma_{\mu}}=c\right\}
\]
is uncountable (all we \emph{need} is: of cardinality at least $2$).

One verifies easily that for each $n,$%
\[
\mathrm{St}_{W}(n)=\overline{W^{e^{n}}}%
\]
where $e=30$ is the exponent of $\mathrm{Alt}(5)$; thus $\mathrm{St}_{W}(n)$
is a topologically characteristic subgroup of $W$.

Let $\mu\neq\nu\in X$. Then%
\[
a(\mu)^{\sigma_{\mu}\sigma_{\nu}^{-1}}=a(\nu).
\]
Now for some $n$ we have $\mu_{n}\neq\nu_{n}$. Say $a(\mu)_{n}=\alpha$ and
$a(\nu)_{n}=\beta$. The continuous automorphism $\sigma_{\mu}\sigma_{\nu}%
^{-1}$ of $W$ fixes both $\mathrm{St}_{W}(n)$ and $\mathrm{St}_{W}(n-1),$ and
therefore induces an automorphism $\tau$ on the quotient
\[
W_{n}=\mathrm{Alt}(5)^{(5^{n})}\rtimes W_{n-1}%
\]
sending the coset of $a(\mu)$ to that of $a(\nu):$%
\[
(1,\ldots,1,\alpha,1,1,1,1)\cdot u\overset{\tau}{\longmapsto}(1,\ldots
,1,\beta,1,1,1,1)\cdot v
\]
in an obvious notation (here, $u$ and $v$ lie in the stabilizer of the point
$5^{n}-4$). This now implies that%
\[
(1,\ldots,1,\beta,1,1,1,1)=(\ast,\ldots,\ast,\alpha^{z},\ast,\ldots,\ast)
\]
for some automorphism $z$ of $\mathrm{Alt}(5)$. This is impossible since
$\alpha$ and $\beta$ have coprime orders.

\end{document}